\documentclass[12pt,reqno]{amsart}
\usepackage{amsthm, amscd, amsfonts, amssymb, graphicx, color}
\usepackage[bookmarksnumbered, colorlinks, plainpages,linkcolor=green,anchorcolor=blue,citecolor=blue,urlcolor=blue]{hyperref}
\usepackage{mathrsfs}
\usepackage{amssymb}
\usepackage{enumerate}
\usepackage{exscale}
\usepackage{relsize}
\usepackage{pdfsync}
\numberwithin{equation}{section}

\makeatletter
\@namedef{subjclassname@2020}{%
  \textup{2020} Mathematics Subject Classification}
\makeatother

\usepackage[includemp,body={398pt,550pt},footskip=30pt,%
            marginparwidth=60pt,marginparsep=10pt]{geometry}

\newtheorem{theorem}{Theorem}[section]
\newtheorem{lemma}[theorem]{Lemma}
\newtheorem{proposition}[theorem]{Proposition}

\theoremstyle{definition}

\newtheorem{remark}[theorem]{Remark}
\numberwithin{equation}{section}

\newcommand{\A}{\mathcal{A}}

\newcommand{\s}{\mathbb{S}}
\newcommand{\R}{\mathcal{R}}

\newcommand{\C}{\mathbb{C}}

\newcommand{\x}{\mathbf{x}}
\newcommand{\y}{\mathbf{y}}
\newcommand{\op}{\mathrm{op}}

\begin{document}
\title[Fock projections on vector-valued $L^p$-spaces]{Fock projections on vector-valued $L^p$-spaces with matrix weights}
\date{October 25, 2025.
}
\author[J. Chen, M. Wang]{Jiale Chen and Maofa Wang}
\address{Jiale Chen, School of Mathematics and Statistics, Shaanxi Normal University, Xi'an 710119, China.}
\email{jialechen@snnu.edu.cn}
\address{Maofa Wang, School of Mathematics and Statistics, Wuhan University, Wuhan 430072, China.}
\email{mfwang.math@whu.edu.cn}

\thanks{This work was supported by National Natural Science Foundation of China (Nos. 12501170, 12571145). }



\begin{abstract}
  \noindent In this paper, we characterize the $d\times d$ matrix weights $W$ on $\mathbb{C}^n$ such that the Fock projection $P_{\alpha}$ is bounded on the vector-valued spaces $L^p_{\alpha,W}(\mathbb{C}^n;\mathbb{C}^d)$ induced by $W$ and the Gaussian measures. It is proved that for $1\leq p\leq\infty$, the Fock projection $P_{\alpha}$ is bounded on $L^p_{\alpha,W}(\mathbb{C}^n;\mathbb{C}^d)$ if and only if $W$ satisfies a restricted $\mathcal{A}_p$-condition. Our result is new even in the scalar setting at the endpoint $p=\infty$.
\end{abstract}
\maketitle


\section{Introduction}
\allowdisplaybreaks[4]

The purpose of this paper is to establish some weighted norm inequalities for the Fock projections on vector-valued $L^p$-spaces with matrix weights. We start from the basic definitions. Fix positive integers $n$ and $d$. A $d\times d$ matrix-valued function $W$ on $\C^n$ is said to be a matrix weight if
\begin{enumerate}
  \item [(i)] $W(z)$ is a positive and invertible operator on $\C^d$ for almost every $z\in\C^n$;
  \item [(ii)] both $W$ and $W^{-1}$ are locally integrable on $\C^n$.
\end{enumerate}
For such a weight $W$ and $1\leq p<\infty$, $\alpha>0$, the weighted space $L^p_{\alpha,W}(\C^n;\C^d)$ consists of $\C^d$-valued measurable functions $f$ on $\C^n$ such that
$$\|f\|^p_{L^p_{\alpha,W}(\C^n;\C^d)}:=\int_{\C^n}\left|W(z)f(z)\right|^p
e^{-\frac{p\alpha}{2}|z|^2}dv(z)<\infty,$$
where $dv$ is the Lebesgue measure on $\C^n$ and $|\cdot|$ is the usual Hilbert space norm on $\C^d$ (or $\C^n$). In the case $p=\infty$, the weighted space $L^{\infty}_{\alpha,W}(\C^n;\C^d)$ of $\C^d$-valued measurable functions $f$ is defined by the following norm:
$$\|f\|_{L^{\infty}_{\alpha,W}(\C^n;\C^d)}:=\mathrm{ess}\sup_{z\in\C^n}|W(z)f(z)|e^{-\frac{\alpha}{2}|z|^2}.$$
Note that the above weighted space $L^p_{\alpha,W}(\C^n;\C^d)$ is slightly different from the one in \cite{CW23,Is14}. The reason that we define $L^p_{\alpha,W}(\C^n;\C^d)$ in this way is that we can treat the endpoint case $p=\infty$ uniformly. It is easy to see that the subspace of $L^p_{\alpha,W}(\C^n;\C^d)$ consisting of $\C^d$-valued entire functions is closed. We denote this subspace by $F^p_{\alpha,W}(\C^n;\C^d)$ and call it weighted Fock space. If $W\equiv E_d$, the $d\times d$ identity matrix, then these spaces are denoted simply by $L^p_{\alpha}(\C^n;\C^d)$ and $F^p_{\alpha}(\C^n;\C^d)$.

It is well-known that the Fock projection $P_{\alpha}$, i.e. the orthogonal projection from $L^2_{\alpha}(\C^n;\C)$ onto the Fock space $F^2_{\alpha}(\C^n;\C)$, is an integral operator given by
$$P_{\alpha}f(z):=\int_{\C^n}f(u)\overline{K^{\alpha}_z(u)}d\lambda_{\alpha}(u), \quad z\in\C^n,$$
where $d\lambda_{\alpha}(u):=\left(\frac{\alpha}{\pi}\right)^ne^{-\alpha|u|^2}dv(u)$ is the Gaussian measure and $K^{\alpha}_z(u):=e^{\alpha\langle u,z\rangle}$ denotes the reproducing kernel of $F^2_{\alpha}(\C^n;\C)$ at $z$. Dostani\'{c} and Zhu \cite{DZ} proved that for all $1\leq p\leq\infty$, $P_{\alpha}$ is bounded on $L^p_{\alpha}(\C^n;\C)$.

A longstanding theme in harmonic analysis is the weighted theory of singular integral operators, which was initiated by Muckenhoupt \cite{Mu72}. It was proved in \cite{Mu72} that for the scalar-valued setting $d=1$, the Hardy--Littlewood maximal operator is bounded on $L^p(w)$ $(1<p<\infty)$ if and only if the weight $w$ satisfies what we call the Muckenhoupt $A_p$-condition today. Later, it was shown that the same condition also characterizes the boundedness of the Riesz projection (equivalently, the Hilbert transform) on $L^p(w)$; see \cite{CF74,HMW}. One direction to extend these results is to consider the vector-valued spaces with matrix weights. Treil and Volberg \cite{TV97} introduced the $A_2$-condition for matrix weights and proved that the Riesz projection is bounded on $L^2(W)$ if and only if the matrix weight $W$ satisfies the $A_2$-condition. Nazarov--Tre\u{\i}l' \cite{NT96} and Volberg \cite{Vo97} introduced the matrix $A_p$-condition and generalized Treil--Volberg's result to all $1<p<\infty$. Later, Christ and Goldberg \cite{CG01,Go03} studied the matrix weighted Hardy--Littlewood maximal operators and used them to establish the boundedness of a class of singular integral operators on $L^p(W)$ for matrix $A_p$-weight $W$. It was pointed out in \cite{Si61,TV97,Vo97} that the matrix weighted estimates for the Riesz projection is closely related to the theory of stationary processes, the invertibility of Toeplitz operators with matrix-valued symbols and a two-weight estimate for the Hilbert transform. Therefore, the matrix weighted theory of singular integral operators and related problems have drawn widespread attention recently; see \cite{BPW,BC,CIMPR,DPTV,IKP,LP,NPTV,Vu} and the references therein. In particular, Cruz-Uribe et al. \cite{CIMPR} established some matrix weighted weak (1,1) type estimates for the Hardy–Littlewood maximal function and Calder\'{o}n--Zygmund operators when the matrix weight satisfies the $A_1$-condition introduced by Frazier and Roudenko \cite{FR04}.

Similarly to the case of Calder\'{o}n--Zygmund operators, one can consider the weighted theory of the Fock projections. To determine the bounded and invertible products of Toeplitz operators on the Fock spaces $F^p_{\alpha}(\C^n;\C)$, Isralowitz \cite{Is14} proved that for $1<p<\infty$, the Fock projection $P_{\alpha}$ is bounded on the scalar-valued space $L^p_{\alpha,w}(\C^n;\C)$ if and only if $w$ satisfies the following restricted $\A_p$-condition: for some (or any) $r>0$,
$$\sup_{Q\subset\C^n:l(Q)=r}\left(\frac{1}{v(Q)}\int_Qw^pdv\right)^{1/p}\left(\frac{1}{v(Q)}\int_{Q}w^{-p'}dv\right)^{1/p'}<\infty.$$
Here and in the sequel, $Q$ denotes a cube in $\C^n$ with sides parallel to the coordinate axes whose side length is denoted by $l(Q)$, and $p'=p/(p-1)$ is the conjugate exponent of $p$. Later, Cascante, F\`{a}brega and Pel\'{a}ez \cite{CFP} proved that $P_{\alpha}$ is bounded on $L^1_{\alpha,w}(\C;\C)$ if and only if the weight $w$ satisfies the following restricted $\A_1$-condition: for some (or any) $r>0$,
$$\sup_{Q\subset\C:l(Q)=r}\frac{\int_Qwdv}{v(Q)\mathrm{ess}\inf_{u\in Q}w(u)}<\infty.$$

Motivated by the aforementioned works, the authors \cite{CW23} recently characterized the matrix weights $W$ on $\C$ such that the Fock projection $P_{\alpha}$ is bounded on the Hilbert spaces $L^2_{\alpha,W}(\C;\C^d)$. It was proved that $P_{\alpha}$ is bounded on $L^2_{\alpha,W}(\C;\C^d)$ if and only if the matrix weight $W$ satisfies the following $\A_{2,r}$-condition for some (or any) $r>0$:
\begin{equation}\label{A2r}
\sup_{Q\subset\C:l(Q)=r}\left\|\left(\frac{1}{v(Q)}\int_{Q}W^2dv\right)^{1/2}
\left(\frac{1}{v(Q)}\int_{Q}W^{-2}dv\right)^{1/2}\right\|_{\op}<\infty.
\end{equation}
Here and in the sequel, for a $d\times d$ matrix $M$, $\|M\|_{\op}$ denotes its operator norm acting on $\C^d$. In this paper, we are going to extend this result to the full range $1\leq p\leq\infty$ and to the matrix weights defined on $\C^n$. Our main result establishes the Fock space analogue of \cite{Go03,Vo97}.

As stated in \cite{Vo97}, the matrix $\A_p$-condition does not have the form as in \eqref{A2r} if $p\neq 2$. We will follow the line in \cite{Vo97} to define some $\A_{p,r}$-condition for matrix weights. Then we show that for $1\leq p\leq\infty$, the Fock projection $P_{\alpha}$ is bounded on $L^{p}_{\alpha,W}(\C^n;\C^d)$ if and only if $W$ satisfies the $\A_{p,r}$-condition for some (or any) $r>0$; see Theorem \ref{main}. More precisely, via a class of integral operators induced by the normalized reproducing kernels of $F^2_{\alpha}(\C^n;\C)$, we show that if $P_{\alpha}$ is bounded on $L^{p}_{\alpha,W}(\C^n;\C^d)$, then the weight $W$ satisfies the $\A_{p,r}$-condition for any $r>0$. Conversely, based on some estimates for $\A_{p,r}$-weights and a duality argument, we prove that if $W$ satisfies the $\A_{p,r}$-condition for some $r>0$, then the matrix weighted maximal Fock projection $P^+_{\alpha,W}$, defined by
$$P^+_{\alpha,W}(f)(z):=
\int_{\C^n}\left|W(z)W^{-1}(u)f(u)\right|\left|K^{\alpha}_z(u)\right|d\lambda_{\alpha}(u),$$
is bounded from $L^p_{\alpha}(\C^n;\C^d)$ into $L^p_{\alpha}(\C^n;\C)$, which implies the boundedness of $P_{\alpha}$ on $L^p_{\alpha,W}(\C^n;\C^d)$. In particular, at the endpoint $p=\infty$, our result is new even in the scalar setting.

The rest part of this paper is organized as follows. In Section \ref{pre}, we recall some preliminary results and introduce the restricted $\A_p$-condition for matrix weights. Then we state and prove our main result in Section \ref{ma}.

Throughout the paper, the notation $\phi_1\lesssim \phi_2$ (or $\phi_2\gtrsim \phi_1$) means that there exists a nonessential constant $c>0$ such that $\phi_1\leq c\phi_2$. If $\phi_1\lesssim \phi_2\lesssim \phi_1$, then we write $\phi_1\asymp \phi_2$. For a subset $E\subset\C^n$, we use $\chi_E$ to denote the characteristic function of $E$. Given $z\in\C^n$ and $r>0$, $Q_r(z)$ denotes the cube centered at $z$ with side length $l(Q)=r$. Finally, we write $p'$ for the conjugate exponent of $p$ for $1\leq p\leq\infty$, i.e. $\frac{1}{p}+\frac{1}{p'}=1$. In particular, $1'=\infty$ and $\infty'=1$.

\section{Preliminaries and definitions}\label{pre}

In this section, we recall some preliminaries and introduce the restricted $\A_p$-condition for matrix weights.

Let $\alpha>0$ and $W$ be a $d\times d$ matrix weight. We use $\langle\cdot,\cdot\rangle_{\alpha}$ to denote the pairing defined as follows:
$$\langle f,g\rangle_{\alpha}:=\int_{\C^n}\langle f(z),g(z)\rangle e^{-\alpha|z|^2}dv(z),$$
where $f$ and $g$ are $\C^d$-valued measurable functions. It is well-known that, with respect to the pairing $\langle\cdot,\cdot\rangle_{\alpha}$, the dual space of $L^p_{\alpha,W}(\C^n;\C^d)$ $(1\leq p<\infty)$ can be represented as
\begin{equation}\label{dual-p}
\left(L^p_{\alpha,W}(\C^n;\C^d)\right)^*=L^{p'}_{\alpha,W^{-1}}(\C^n;\C^d).
\end{equation}

We now introduce the restricted $\A_p$-conditions for matrix weights. To this end, we consider the more general situation of norm-valued functions. Let $z\mapsto \rho_z,\ z\in\C^n$ be a function whose values are Banach space norms on $\C^d$. We assume this function to be measurable in the sense that for any $\x\in\C^d$, the function $z\mapsto \rho_z(\x)$ is measurable. For convenience, we will say the norm-valued function $z\mapsto \rho_z$ to be a metric and denote it by $\rho$.

Let $\rho$ be a metric. For $z\in\C^n$, the dual norm $(\rho_z)^*$ of $\rho_z$ is defined by
$$(\rho_z)^*(\x):=\sup_{\y\in\C^d\setminus\{0\}}\frac{|\langle\x,\y\rangle|}{\rho_z(\y)},\quad \x\in\C^d.$$
Since $\C^d$ is reflexive, we have $(\rho_z)^{**}=\rho_z$. The dual metric $\rho^*$ is defined pointwisely by $\rho^*_z=(\rho_z)^*$. For $1\leq p<\infty$ and a cube $Q\subset\C^n$, consider the norms
$$\rho_{p,Q}(\x):=\left(\frac{1}{v(Q)}\int_Q\left(\rho_z(\x)\right)^pdv(z)\right)^{1/p},\quad \x\in\C^d$$
and
$$\rho_{\infty,Q}(\x):=\mathrm{ess}\sup_{z\in Q}\rho_z(\x),\quad \x\in\C^d.$$
The following lemma was proved in \cite[Proposition 1.1]{Vo97} for the case $1<p<\infty$. The case $p=1$ or $p=\infty$ is similar and we omit the proof here.
\begin{lemma}\label{geq}
Let $1\leq p\leq\infty$ and $\rho$ be a metric. Then for any $\x\in\C^d$ and any cube $Q\subset\C^n$,
$$\rho^*_{p',Q}(\x)\geq\left(\rho_{p,Q}\right)^*(\x).$$
\end{lemma}

For $1\leq p\leq\infty$ and $r>0$, a metric $\rho$ is called an $\A_{p,r}$-metric if there exists some constant $C>0$ such that the opposite inequality
$$\rho^*_{p',Q}(\x)\leq C\left(\rho_{p,Q}\right)^*(\x)$$
holds for any $\x\in\C^d$ and any cube $Q\subset\C^n$ with $l(Q)=r$. The $\A_{p,r}$-constant of $\rho$, denoted by $[\rho]_{\A_{p,r}}$, is the least suitable constant $C$. Let $W$ be a $d\times d$ matrix weight. For $1\leq p\leq\infty$ and $r>0$, we say that $W$ is an $\A_{p,r}$-weight if the metric $\rho$ defined by
$$\rho_{z}(\x)=\left|W(z)\x\right|,\quad \x\in\C^d,\quad z\in\C^n$$
is an $\A_{p,r}$-metric. Moreover, we write $[W]_{\A_{p,r}}=[\rho]_{\A_{p,r}}$.

It was pointed out in \cite[p. 449]{Vo97} (see also \cite[Proposition 1.2]{Go03}) that for a Banach space norm $\rho_z$ on $\C^d$, there exists a $d\times d$ self-adjoint matrix $M_z$, which is a positive and invertible operator on $\C^d$, such that
$$\rho_z(\x)\leq|M_z\x|\leq\sqrt{d}\cdot\rho_z(\x),\quad \x\in\C^d.$$
Then the dual norm $(\rho_z)^*$ satisfies
$$\left|M^{-1}_z\x\right|\leq(\rho_z)^*(\x)\leq\sqrt{d}\cdot\left|M^{-1}_z\x\right|,\quad \x\in\C^d.$$
We will call the matrix $M_z$ the reducing operator of $\rho_z$. It is now possible to state the $\A_{p,r}$-condition in terms of the reducing operators. Let $1\leq p\leq\infty$ and $\rho$ be a metric. For any cube $Q\subset\C^n$, we use $\R_Q$ and $\R^{\star}_Q$ to denote the reducing operators of $\rho_{p,Q}$ and $\rho^*_{p',Q}$, respectively. Consequently,
\begin{equation}\label{reduce}
	\rho_{p,Q}(\x)\leq|\R_Q\x|\leq\sqrt{d}\cdot\rho_{p,Q}(\x),\quad \x\in\C^d,
\end{equation}
and
\begin{equation}\label{reduce-s}
	\rho^*_{p',Q}(\x)\leq|\R^{\star}_Q\x|\leq\sqrt{d}\cdot\rho^*_{p',Q}(\x),\quad \x\in\C^d.
\end{equation}
Combining these inequalities with the definition of the $\A_{p,r}$-metric, we obtain that $\rho$ is an $\A_{p,r}$-metric if and only if
$$\sup_{Q\subset\C^n:l(Q)=r}\|\R_Q\R^{\star}_Q\|_{\op}<\infty.$$
Moreover,
\begin{equation}\label{R-Apr}
	[\rho]_{\A_{p,r}}\leq\sup_{Q\subset\C^n:l(Q)=r}\|\R_Q\R^{\star}_Q\|_{\op}\leq d[\rho]_{\A_{p,r}}.
\end{equation}
Based on this characterization, the $\A_{p,r}$ matrix weights can be described by integral averages that is more like the scalar case (see \cite[Lemma 1.3]{Ro03} and \cite[Proposition 6.5]{BC}): for $1<p<\infty$, $W\in\A_{p,r}$ if and only if
$$\sup_{Q\subset\C^n:l(Q)=r}\left(\frac{1}{v(Q)}\int_Q\left(\frac{1}{v(Q)}\int_Q\left\|W(z)W^{-1}(u)\right\|^{p'}_{\op}dv(u)
\right)^{\frac{p}{p'}}dv(z)\right)^{\frac{1}{p}}<\infty;$$
$W\in\A_{1,r}$ if and only if
$$\sup_{Q\subset\C^n:l(Q)=r}\mathrm{ess}\sup_{z\in Q}\frac{1}{v(Q)}\int_Q\left\|W^{-1}(z)W(u)\right\|_{\op}dv(u)<\infty;$$
and, $W\in\A_{\infty,r}$ if and only if
$$\sup_{Q\subset\C^n:l(Q)=r}\mathrm{ess}\sup_{z\in Q}\frac{1}{v(Q)}\int_Q\left\|W(z)W^{-1}(u)\right\|_{\op}dv(u)<\infty.$$

\begin{remark}
Let $W$ be a matrix weight and let the metric $\rho$ be defined by $\rho_{z}(\x)=\left|W(z)\x\right|$. Then the reducing operators of $\rho_{2,Q}$ and $\rho^*_{2,Q}$ can be calculated as follows:
$$\R_{Q}=\left(\frac{1}{v(Q)}\int_QW^2dv\right)^{1/2}
\quad\mathrm{and}\quad
\R^{\star}_{Q}=\left(\frac{1}{v(Q)}\int_QW^{-2}dv\right)^{1/2},$$
which gives the $\A_{2,r}$-condition \eqref{A2r}.
\end{remark}

\section{The main result}\label{ma}

In this section, we state and prove our main result, which characterizes the boundedness of $P_{\alpha}$ on $L^p_{\alpha,W}(\C^n;\C^d)$ for all $1\leq p\leq\infty$. Moreover, the corresponding norm estimate for $P_{\alpha}$ is established. Recall that the Fock projection $P_{\alpha}$ is defined by
$$P_{\alpha}f(z)=\int_{\C^n}f(u)\overline{K^{\alpha}_z(u)}d\lambda_{\alpha}(u), \quad z\in\C^n,$$
and the matrix weighted maximal Fock projection $P^+_{\alpha,W}$ is defined by
$$P^+_{\alpha,W}(f)(z)=
\int_{\C^n}\left|W(z)W^{-1}(u)f(u)\right|\left|K^{\alpha}_z(u)\right|d\lambda_{\alpha}(u),\quad z\in\C^n.$$

Our main result reads as follows.

\begin{theorem}\label{main}
Let $\alpha>0$, $1\leq p\leq\infty$, and let $W$ be a $d\times d$ matrix weight on $\C^n$. The following conditions are equivalent:
\begin{enumerate}
	\item [(a)] $P_{\alpha}$ is bounded on $L^p_{\alpha,W}(\C^n;\C^d)$;
	\item [(b)] $P^+_{\alpha,W}:L^p_{\alpha}(\C^n;\C^d)\to L^p_{\alpha}(\C^n;\C)$ is bounded;
	\item [(c)] $W$ is an $\A_{p,r}$-weight for any $r>0$;
	\item [(d)] $W$ is an $\A_{p,r}$-weight for some $r>0$.
\end{enumerate}
Moreover, for any fixed $r>0$, there exists $c=c(\alpha,p,r,n)>0$ such that
\begin{equation*}\label{Pnorm}
\left(\frac{\alpha r^2}{\pi}\right)^ne^{-n\alpha r^2}[W]_{\A_{p,r}}^{1/2}
\leq\|P_{\alpha}\|\leq\left\|P^+_{\alpha,W}\right\|
\leq cd^{c_p}[W]_{\A_{p,r}}^{c(1+\log[W]_{\A_{p,r}})},
\end{equation*}
where $c_p=7/2$ for $1<p<\infty$, and $c_1=c_{\infty}=9/2$.
\end{theorem}

As a byproduct of Theorem \ref{main}, we know that the class of $\A_{p,r}$-weight (more generally, the class of $\A_{p,r}$-metric) is actually independent of the choice of $r$. In fact, we will give a direct and quantitative proof of this fact; see Proposition \ref{direct}.

For the proof of Theorem \ref{main}, we first note that if $P^+_{\alpha,W}:L^p_{\alpha}(\C^n;\C^d)\to L^p_{\alpha}(\C^n;\C)$ is bounded, then for $f\in L^p_{\alpha,W}(\C^n;\C^d)$,
\begin{align*}
\|P_{\alpha}f\|^p_{L^p_{\alpha,W}(\C^n;\C^d)}
&=\int_{\C^n}\left|W(z)\int_{\C^n}f(u)\overline{K^{\alpha}_z(u)}d\lambda_{\alpha}(u)\right|^pe^{-\frac{p\alpha}{2}|z|^2}dv(z)\\
&\leq\int_{\C^n}\left(\int_{\C^n}\left|W(z)f(u)\right||K^{\alpha}_z(u)|d\lambda_{\alpha}(u)\right)^p
  e^{-\frac{p\alpha}{2}|z|^2}dv(z)\\
&=\left\|P^+_{\alpha,W}\left(Wf\right)\right\|^p_{L^p_{\alpha}(\C^n;\C)}\\
&\leq\left\|P^+_{\alpha,W}\right\|^p\left\|Wf\right\|^p_{L^p_{\alpha}(\C^n;\C^d)}\\
&=\left\|P^+_{\alpha,W}\right\|^p\left\|f\right\|^p_{L^p_{\alpha,W}(\C^n;\C^d)}.
\end{align*}
Therefore, the implication (b)$\Longrightarrow$(a) of Theorem \ref{main} holds, and
$$\|P_{\alpha}\|\leq\left\|P^+_{\alpha,W}\right\|.$$
The rest part of the paper is devoted to proving the implications (a)$\Longrightarrow$(c) and (d)$\Longrightarrow$(b).

To establish the implication (a)$\Longrightarrow$(c), we consider a class of integral operators induced by the normalized reproducing kernels. Given $\alpha>0$ and $u\in\C^n$, we use $k^{\alpha}_u$ to denote the normalized reproducing kernel of $F^2_{\alpha}(\C^n;\C)$, that is,
$$k^{\alpha}_u(z)=e^{\alpha\langle z,u\rangle-\frac{\alpha}{2}|u|^2},\quad z\in\C^n.$$
For fixed $r>0$, define an operator $P_{\alpha,u,r}$ for $\C^d$-valued functions $f$ by
$$P_{\alpha,u,r}f=\chi_{Q_r(u)}k^{\alpha}_u\int_{Q_r(u)}f\overline{{k}^{\alpha}_u}d\lambda_{\alpha}.$$
The following proposition reveals the relation between the boundedness of $P_{\alpha}$ and $P_{\alpha,u,r}$.

\begin{proposition}\label{relation}
Let $\alpha,r>0$, $1\leq p\leq\infty$, $u\in\C^n$, and let $W$ be a $d\times d$ matrix weight on $\C^n$. Suppose that $P_{\alpha}$ is bounded on $L^p_{\alpha,W}(\C^n;\C^d)$. Then $P_{\alpha,u,r}$ is bounded on $L^p_{\alpha,W}(\C^n;\C^d)$, and
$$\|P_{\alpha,u,r}\|\leq e^{n\alpha r^2/2}\|P_{\alpha}\|.$$
\end{proposition}
\begin{proof}
Write $Q=Q_r(u)$ temporarily to save the notation. Since $P_{\alpha}$ is bounded on $L^p_{\alpha,W}(\C^n;\C^d)$, it is clear that for any $f\in L^p_{\alpha,W}(\C^n;\C^d)$,  
\begin{equation}\label{clear}
\|\chi_QP_{\alpha}(\chi_Qf)\|_{L^p_{\alpha,W}(\C^n;\C^d)}
\leq\|P_{\alpha}\|\|f\|_{L^p_{\alpha,W}(\C^n;\C^d)}.
\end{equation}
We now estimate the norm of $P_{\alpha,u,r}f-\chi_QP_{\alpha}(\chi_Qf)$. For any $z\in\C$,
\begin{align*}
	&P_{\alpha,u,r}f(z)-\chi_Q(z)P_{\alpha}(\chi_Qf)(z)\\
	=&\,\chi_Q(z)\int_Qf(\zeta)e^{\alpha \langle z,\zeta\rangle}\left(\sum_{k=1}^{\infty}\frac{(-1)^k\alpha^k}{k!}
	\langle u-z,u-\zeta\rangle^k\right)d\lambda_{\alpha}(\zeta);
\end{align*}
see the proof of \cite[Proposition 2.2]{CW23}. Note that for any $k\geq1$,
\begin{align*}
\langle u-z,u-\zeta\rangle^k&=\left(\sum_{j=1}^n(u_j-z_j)\overline{(u_j-\zeta_j)}\right)^k\\
	&=\sum_{\substack{k_1,k_2,\cdots,k_n\geq0\\k_1+\cdots +k_n=k}}\frac{k!}{k_1!\cdots k_n!}
	  \prod_{j=1}^n(u_j-z_j)^{k_j}(\overline{u_j-\zeta_j})^{k_j}.
\end{align*}
For $u\in\C^n$ and $k_1,\cdots,k_n\geq0$, write $\tau_u^{k_1,\cdots,k_n}(z)=\prod_{j=1}^n(u_j-z_j)^{k_j}$. Then
{\footnotesize
\begin{align*}
&P_{\alpha,u,r}f(z)-\chi_Q(z)P_{\alpha}(\chi_Qf)(z)\\
=&\chi_Q(z)\int_Qf(\zeta)e^{\alpha\langle z,\zeta\rangle}
\left(\sum_{k=1}^{\infty}\frac{(-1)^k\alpha^k}{k!}
\sum_{\substack{k_1,k_2,\cdots,k_n\geq0\\k_1+\cdots +k_n=k}}\frac{k!}{k_1!\cdots k_n!}
\tau_u^{k_1,\cdots,k_n}(z)\overline{\tau_u^{k_1,\cdots,k_n}(\zeta)}\right)d\lambda_{\alpha}(\zeta)\\
=&\sum_{k=1}^{\infty}\frac{(-1)^k\alpha^k}{k!}\sum_{\substack{k_1,k_2,\cdots,k_n\geq0\\k_1+\cdots +k_n=k}}\frac{k!}{k_1!\cdots k_n!}
\chi_Q(z)\tau_u^{k_1,\cdots,k_n}(z)
\int_Q\overline{\tau_u^{k_1,\cdots,k_n}(\zeta)}f(\zeta)e^{\alpha\langle z,\zeta\rangle}d\lambda_{\alpha}(\zeta)\\
=&\sum_{k=1}^{\infty}\frac{(-1)^k\alpha^k}{k!}\sum_{\substack{k_1,k_2,\cdots,k_n\geq0\\k_1+\cdots +k_n=k}}\frac{k!}{k_1!\cdots k_n!}
\chi_Q(z)\tau_u^{k_1,\cdots,k_n}(z)P_{\alpha}\left(\chi_Q\overline{\tau_u^{k_1,\cdots,k_n}}f\right)(z).
\end{align*}
}
Noting that for any $k\geq1$ and any $k_1,\cdots,k_n\geq0$ with $k_1+\cdots+k_n=k$,
$$\left|\tau_u^{k_1,\cdots,k_n}(z)\right|=\prod_{j=1}^n|u_j-z_j|^{k_j}\leq2^{-k/2}r^k,\quad \forall z\in Q,$$
we have
$$\left\|\chi_Q\tau_u^{k_1,\cdots,k_n}f\right\|_{L^p_{\alpha,W}(\C^n;\C^d)}\leq2^{-k/2}r^{k}\|f\|_{L^p_{\alpha,W}(\C^n;\C^d)},$$
which implies that
\begin{align*}
&\left\|P_{\alpha,u,r}f-\chi_QP_{\alpha}(\chi_Qf)\right\|_{L^p_{\alpha,W}(\C^n;\C^d)}\\
\leq&\sum_{k=1}^{\infty}\frac{\alpha^k}{k!}\sum_{\substack{k_1,k_2,\cdots,k_n\geq0\\k_1+\cdots +k_n=k}}\frac{k!}{k_1!\cdots k_n!}
  \left\|\chi_Q\tau_u^{k_1,\cdots,k_n}P_{\alpha}\left(\chi_Q\overline{\tau_u^{k_1,\cdots,k_n}}f\right)\right\|_{L^p_{\alpha,W}(\C^n;\C^d)}\\
\leq&\sum_{k=1}^{\infty}\frac{\alpha^kr^k}{2^{k/2}k!}
  \sum_{\substack{k_1,k_2,\cdots,k_n\geq0\\k_1+\cdots +k_n=k}}\frac{k!}{k_1!\cdots k_n!}
  \left\|P_{\alpha}\left(\chi_Q\overline{\tau_u^{k_1,\cdots,k_n}}f\right)\right\|_{L^p_{\alpha,W}(\C^n;\C^d)}\\
\leq&\|P_{\alpha}\|\sum_{k=1}^{\infty}\frac{\alpha^kr^k}{2^{k/2}k!}
  \sum_{\substack{k_1,k_2,\cdots,k_n\geq0\\k_1+\cdots +k_n=k}}\frac{k!}{k_1!\cdots k_n!}
  \left\|\chi_Q\overline{\tau_u^{k_1,\cdots,k_n}}f\right\|_{L^p_{\alpha,W}(\C^n;\C^d)}\\
\leq&\|P_{\alpha}\|\sum_{k=1}^{\infty}\frac{\alpha^kr^{2k}}{2^kk!}
  \sum_{\substack{k_1,k_2,\cdots,k_n\geq0\\k_1+\cdots +k_n=k}}\frac{k!}{k_1!\cdots k_n!}\|f\|_{L^p_{\alpha,W}(\C^n;\C^d)}\\
=&\left(e^{\frac{n\alpha r^2}{2}}-1\right)\|P_{\alpha}\|\|f\|_{L^p_{\alpha,W}(\C^n;\C^d)}.
\end{align*}
Combining this with \eqref{clear} yields the desired result.
\end{proof}

The following lemma is easy to verify, so we omit the proof.

\begin{lemma}\label{self}
Let $\alpha,r>0$, $1\leq p\leq\infty$, $u\in\C^n$, and let $W$ be a $d\times d$ matrix weight on $\C^n$ such that $P_{\alpha}$ is bounded on $L^p_{\alpha,W}(\C^n;\C^d)$. Then for $f\in L^p_{\alpha,W}(\C^n;\C^d)$ and $g\in L^{p'}_{\alpha,W^{-1}}(\C^n;\C^d)$,
$$\langle P_{\alpha,u,r}f,g\rangle_{\alpha}=\langle f,P_{\alpha,u,r}g\rangle_{\alpha}.$$
\end{lemma}

\begin{lemma}\label{ness-key}
Let $\alpha,r>0$, $1\leq p\leq\infty$, $u\in\C^n$, and let $W$ be a $d\times d$ matrix weight on $\C^n$ such that $P_{\alpha}$ is bounded on $L^p_{\alpha,W}(\C^n;\C^d)$. Let the function $f$ be defined by
$$f=c\chi_{Q_r(u)}k^{\alpha}_u\x,$$
where $c\in\C$ and $\x\in\C^d$. Then
$$\sup_{\y\in\C^d\setminus\{0\}}\frac{|\langle f,\chi_{Q_r(u)}k^{\alpha}_u\y\rangle_{\alpha}|}
		{\|\chi_{Q_r(u)}k^{\alpha}_u\y\|_{L^{p'}_{\alpha,W^{-1}}(\C^n;\C^d)}}
		\geq\left(\frac{\alpha r^2}{\pi}\right)^ne^{-n\alpha r^2}\|P_{\alpha}\|^{-1}\|f\|_{L^p_{\alpha,W}(\C^n;\C^d)}.$$
\end{lemma}
\begin{proof}
It is easy to see that $P_{\alpha,u,r}f=c_{\alpha,u,r}f$, where $c_{\alpha,u,r}=\int_{Q_r(u)}|k^{\alpha}_u|^2d\lambda_{\alpha}$. Suppose first that $1\leq p<\infty$. Using the duality \eqref{dual-p} and Lemma \ref{self}, we obtain that
\begin{align*}
\|f\|_{L^p_{\alpha,W}(\C^n;\C^d)}&=\sup_{\phi}|\langle f,\phi\rangle_{\alpha}|\\
&=\sup_{\phi}\left|\left\langle P_{\alpha,u,r}f,\frac{1}{c_{\alpha,u,r}}\phi\right\rangle_{\alpha}\right|\\
&=\sup_{\phi}\left|\left\langle f,\frac{1}{c_{\alpha,u,r}}P_{\alpha,u,r}\phi\right\rangle_{\alpha}\right|,
\end{align*}
where the supremum is taken over $\phi\in L^{p'}_{\alpha,W^{-1}}(\C^n;\C^d)$ with $\|\phi\|\leq1$. Write $\y_{\phi}=\int_{Q_r(u)}\phi\overline{k^{\alpha}_u}d\lambda_{\alpha}$. Then $P_{\alpha,u,r}\phi=\chi_{Q_r(u)}k^{\alpha}_u\y_{\phi}$. Therefore,
$$
\|f\|_{L^p_{\alpha,W}(\C^n;\C^d)}
=\sup_{\phi:\y_{\phi}\neq0}\left|\left\langle f,\frac{1}{c_{\alpha,u,r}}\chi_{Q_r(u)}k^{\alpha}_u\y_{\phi}\right\rangle_{\alpha}\right|.
$$
By Proposition \ref{relation} and Lemma \ref{self}, we know that the operator $P_{\alpha,u,r}$ is bounded on $L^{p'}_{\alpha,W^{-1}}(\C^n;\C^d)$, which implies that
\begin{align*}
\left\|\frac{1}{c_{\alpha,u,r}}\chi_{Q_r(u)}k^{\alpha}_u\y_{\phi}\right\|_{L^{p'}_{\alpha,W^{-1}}(\C^n;\C^d)}
&=\frac{1}{c_{\alpha,u,r}}\left\|P_{\alpha,u,r}\phi\right\|_{L^{p'}_{\alpha,W^{-1}}(\C^n;\C^d)}\\
&\leq\frac{1}{c_{\alpha,u,r}}\|P_{\alpha,u,r}\|\\
&\leq\frac{1}{c_{\alpha,u,r}}e^{\frac{n\alpha r^2}{2}}\|P_{\alpha}\|.
\end{align*}
It is clear that $c_{\alpha,u,r}=\left(\frac{\alpha}{\pi}\right)^n\int_{Q_r(u)}e^{-\alpha|z-u|^2}dv(z)
\geq \left(\frac{\alpha r^2}{\pi}\right)^ne^{-n\alpha r^2/2}$. Hence
$$\left\|\frac{1}{c_{\alpha,u,r}}\chi_{Q_r(u)}k^{\alpha}_u\y_{\phi}\right\|_{L^{p'}_{\alpha,W^{-1}}(\C^n;\C^d)}
\leq \left(\frac{\pi}{\alpha r^2}\right)^ne^{n\alpha r^2}\|P_{\alpha}\|.$$
Consequently,
\begin{align*}
\|f\|_{L^p_{\alpha,W}(\C^n;\C^d)}
&\leq \left(\frac{\pi}{\alpha r^2}\right)^ne^{n\alpha r^2}\|P_{\alpha}\|
  \sup_{\phi:\y_{\phi}\neq0}\frac{\left|\left\langle f,\chi_{Q_r(u)}k^{\alpha}_u\frac{\y_{\phi}}{c_{\alpha,u,r}}\right\rangle_{\alpha}\right|}
  {\|\chi_{Q_r(u)}k^{\alpha}_u\frac{\y_{\phi}}{c_{\alpha,u,r}}\|_{L^{p'}_{\alpha,W^{-1}}(\C^n;\C^d)}}\\
&\leq \left(\frac{\pi}{\alpha r^2}\right)^ne^{n\alpha r^2}\|P_{\alpha}\|
  \sup_{\y\in\C^d\setminus\{0\}}\frac{\left|\left\langle f,\chi_{Q_r(u)}k^{\alpha}_u\y\right\rangle_{\alpha}\right|}
  {\|\chi_{Q_r(u)}k^{\alpha}_u\y\|_{L^{p'}_{\alpha,W^{-1}}(\C^n;\C^d)}},
\end{align*}
which is exactly what we want. In the case $p=\infty$, the boundedness of $P_{\alpha}$ on $L^{\infty}_{\alpha,W}(\C^n;\C^d)$ together with Proposition \ref{relation} and Lemma \ref{self} implies that $P_{\alpha,u,r}$ is bounded on $L^1_{\alpha,W^{-1}}(\C^n;\C^d)$. Therefore, using the duality $\left(L^1_{\alpha,W^{-1}}(\C^n;\C^d)\right)^*=L^{\infty}_{\alpha,W}(\C^n;\C^d)$, we can obtain the desired result by following the same procedure.
\end{proof}

We are now ready to prove the implication (a)$\Longrightarrow$(c) of Theorem \ref{main}.

\begin{proof}[Proof of Theorem \ref{main}. (a)$\Longrightarrow$(c)]
Fix $r>0$, and let $\rho$ be the metric defined by
$$\rho_z(\x)=\left|W(z)\x\right|,\quad \x\in\C^d,\quad z\in\C^n.$$
It is sufficient to show that $\rho$ is an $\A_{p,r}$-metric. To this end, fix $u\in\C^n$ and $\x\in\C^d$. Define
$$f=\frac{1}{c_{\alpha,u,r}}\chi_{Q_r(u)}k^{\alpha}_u\x.$$
Then $P_{\alpha,u,r}f=\chi_{Q_r(u)}k^{\alpha}_u\x$ and $\x=\int_{Q_r(u)}f\overline{k^{\alpha}_u}d\lambda_{\alpha}$. Combining the boundedness of $P_{\alpha}$ with Proposition \ref{relation} yields that for $1\leq p<\infty$,
\begin{align}\label{low}
\nonumber e^{n\alpha r^2/2}\|P_{\alpha}\|\|f\|_{L^p_{\alpha,W}(\C^n;\C^d)}&\geq\left\|P_{\alpha,u,r}f\right\|_{L^p_{\alpha,W}(\C^n;\C^d)}\\
  &\nonumber=\left(\int_{Q_r(u)}\left|W(z)k^{\alpha}_u(z)\x\right|^pe^{-\frac{p\alpha}{2}|z|^2}dv(z)\right)^{1/p}\\
  &\nonumber=\left(\int_{Q_r(u)}\left|W(z)\x\right|^pe^{-\frac{p\alpha}{2}|z-u|^2}dv(z)\right)^{1/p}\\
  &\nonumber\geq e^{-\frac{n\alpha r^2}{4}}\left(\int_{Q_r(u)}\left|W(z)\x\right|^pdv(z)\right)^{1/p}\\
  &=e^{-\frac{n\alpha r^2}{4}}r^{2n/p}\rho_{p,Q_r(u)}(\x),
\end{align}
and for $p=\infty$,
\begin{align}\label{low1}
\nonumber e^{n\alpha r^2/2}\|P_{\alpha}\|\|f\|_{L^{\infty}_{\alpha,W}(\C^n;\C^d)}
&\geq\left\|P_{\alpha,u,r}f\right\|_{L^{\infty}_{\alpha,W}(\C^n;\C^d)}\\
&\nonumber=\mathrm{ess}\sup_{z\in Q_r(u)}\left|W(z)k^{\alpha}_u(z)\x\right|e^{-\frac{\alpha}{2}|z|^2}\\
&\nonumber=\mathrm{ess}\sup_{z\in Q_r(u)}\left|W(z)\x\right|e^{-\frac{\alpha}{2}|z-u|^2}\\
&\nonumber\geq e^{-\frac{n\alpha r^2}{4}}\mathrm{ess}\sup_{z\in Q_r(u)}\left|W(z)\x\right|\\
&=e^{-\frac{n\alpha r^2}{4}}\rho_{\infty,Q_r(u)}(\x).
\end{align}
Suppose now that $1<p\leq\infty$. Then we have
\begin{align*}
&\left(\rho^*_{p',Q_r(u)}\right)^*(\x)=\sup_{\y\in\C^d\setminus\{0\}}\frac{|\langle \x,\y\rangle|}{\rho^*_{p',Q_r(u)}(\y)}\\
&\qquad=\sup_{\y\in\C^d\setminus\{0\}}
  \frac{\left|\left\langle\int_{Q_r(u)}f(\zeta)\overline{k^{\alpha}_u(\zeta)}d\lambda_{\alpha}(\zeta),\y\right\rangle\right|}
  {\left(\frac{1}{v(Q_r(u))}\int_{Q_r(u)}\left|W^{-1}(z)\y\right|^{p'}dv(z)\right)^{1/p'}}\\
&\qquad\geq r^{2n/p'}e^{-\frac{n\alpha r^2}{4}}\sup_{\y\in\C^d\setminus\{0\}}
  \frac{\left|\left(\frac{\alpha}{\pi}\right)^n\int_{Q_r(u)}\left\langle f(\zeta)\overline{k^{\alpha}_u(\zeta)},\y\right\rangle e^{-\alpha|\zeta|^2}dv(\zeta)\right|}
  {\left(\int_{Q_r(u)}\left|W^{-1}(z)\y\right|^{p'}|k^{\alpha}_u(z)|^{p'}e^{-\frac{p'\alpha}{2}|z|^2}dv(z)\right)^{1/p'}}\\
&\qquad=\left(\frac{\alpha}{\pi}\right)^nr^{2n/p'}e^{-\frac{n\alpha r^2}{4}}\sup_{\y\in\C^d\setminus\{0\}}
  \frac{\left|\langle f,\frac{}{}\chi_{Q_r(u)}k^{\alpha}_u\y\rangle_{\alpha}\right|}
  {\|\chi_{Q_r(u)}k^{\alpha}_u\y\|_{L^{p'}_{\alpha,W^{-1}}(\C^n;\C^d)}},
\end{align*}
which, in conjunction with Lemma \ref{ness-key}, implies that
$$\left(\rho^*_{p',Q_r(u)}\right)^*(\x)
\geq\left(\frac{\alpha}{\pi}\right)^{2n}r^{2n+2n/p'}e^{-\frac{5n\alpha r^2}{4}}\|P_{\alpha}\|^{-1}\|f\|_{L^p_{\alpha,W}(\C^n;\C^d)}.$$
Combining the above inequality with \eqref{low} and \eqref{low1}, we deduce that
$$\left(\rho^*_{p',Q_r(u)}\right)^*(\x)\geq
\left(\frac{\alpha r^2}{\pi}\right)^{2n}e^{-2n\alpha r^2}\|P_{\alpha}\|^{-2}\rho_{p,Q_r(u)}(\x),$$
which is equivalent to
$$\rho^*_{p',Q_r(u)}(\x)\leq\left(\frac{\pi}{\alpha r^2}\right)^{2n}e^{2n\alpha r^2}\|P_{\alpha}\|^2\left(\rho_{p,Q_r(u)}\right)^*(\x).$$
Since $u\in\C^n$ and $\x\in\C^d$ are both arbitrary, we know that $\rho$ is an $\A_{p,r}$-metric. Moreover,
$$[\rho]_{\A_{p,r}}\leq\left(\frac{\pi}{\alpha r^2}\right)^{2n}e^{2n\alpha r^2}\|P_{\alpha}\|^2.$$
Suppose now that $p=1$. Then similarly,
\begin{align*}
	\left(\rho^*_{\infty,Q_r(u)}\right)^*(\x)&=\sup_{\y\in\C^d\setminus\{0\}}\frac{|\langle \x,\y\rangle|}{\rho^*_{\infty,Q_r(u)}(\y)}\\
	&=\sup_{\y\in\C^d\setminus\{0\}}
	\frac{\left|\left\langle\int_{Q_r(u)}f(\zeta)\overline{k^{\alpha}_u(\zeta)}d\lambda_{\alpha}(\zeta),\y\right\rangle\right|}
	{\mathrm{ess}\sup_{z\in Q_r(u)}\left|W^{-1}(z)\y\right|}\\
	&\geq\left(\frac{\alpha}{\pi}\right)^ne^{-\frac{n\alpha r^2}{4}}\sup_{\y\in\C^d\setminus\{0\}}
	\frac{\left|\langle f,\chi_{Q_r(u)}k^{\alpha}_u\y\rangle_{\alpha}\right|}
	{\|\chi_{Q_r(u)}k^{\alpha}_u\y\|_{L^{\infty}_{\alpha,W^{-1}}(\C^n;\C^d)}}\\
	&\geq\left(\frac{\alpha r}{\pi}\right)^{2n}e^{-\frac{5n\alpha r^2}{4}}\|P_{\alpha}\|^{-1}\|f\|_{L^1_{\alpha,W}(\C^n;\C^d)}\\
	&\geq\left(\frac{\alpha r^2}{\pi}\right)^{2n}e^{-2n\alpha r^2}\|P_{\alpha}\|^{-2}\rho_{1,Q_r(u)}(\x),
\end{align*}
which finishes the proof.
\end{proof}

We now turn to the implication (d)$\Longrightarrow$(b) of Theorem \ref{main}. Before proceeding, we establish some estimates for $\A_{p,r}$-metrics. For a cube $Q\subset\C^n$, we use $3Q$ to denote the cube with the same center but with side length $3l(Q)$.
\begin{lemma}\label{3Q}
Let $r>0$, $1\leq p\leq\infty$, and let $\rho$ be an $\A_{p,3r}$-metric. Then for any cube $Q\subset\C^n$ with $l(Q)=r$ and any $\x\in\C^d$,
$$\rho_{p,3Q}(\x)\leq 3^{2n/p'}[\rho]_{\A_{p,3r}}\rho_{p,Q}(\x).$$
\end{lemma}
\begin{proof}
The $\A_{p,3r}$-condition implies that
$$\rho_{p,3Q}(\x)=\sup_{\y\in\C^d\setminus\{0\}}\frac{|\langle \x,\y\rangle|}{\left(\rho_{p,3Q}\right)^*(\y)}
\leq[\rho]_{\A_{p,3r}}\sup_{\y\in\C^d\setminus\{0\}}\frac{|\langle \x,\y\rangle|}{\rho^*_{p',3Q}(\y)}.$$
It is clear that $\rho^*_{p',3Q}(\y)\geq3^{-2n/p'}\rho^*_{p',Q}(\y)$. Hence by Lemma \ref{geq},
\begin{align*}
\rho_{p,3Q}(\x)&\leq3^{2n/p'}[\rho]_{\A_{p,3r}}\sup_{\y\in\C^d\setminus\{0\}}\frac{|\langle \x,\y\rangle|}{\rho^*_{p',Q}(\y)}\\
&\leq3^{2n/p'}[\rho]_{\A_{p,3r}}\sup_{\y\in\C^d\setminus\{0\}}\frac{|\langle \x,\y\rangle|}{\left(\rho_{p,Q}\right)^*(\y)}\\
&=3^{2n/p'}[\rho]_{\A_{p,3r}}\rho_{p,Q}(\x),
\end{align*}
which completes the proof.
\end{proof}

For each $r>0$, we will treat $r\mathbb{Z}^{2n}$ as a subset of $\C^n$ via the canonical identification between $\mathbb{R}^{2n}$ and $\C^n$.

\begin{lemma}\label{nunu'}
Let $r>0$, $1\leq p\leq\infty$, and let $\rho$ be an $\A_{p,3r}$-metric. Then for any $\nu,\nu'\in r\mathbb{Z}^{2n}$,
$$\rho_{p,Q_r(\nu)}(\x)\leq\left(3^{2n}[\rho]_{\A_{p,3r}}\right)^{\frac{\sqrt{2n}}{r}|\nu-\nu'|}\rho_{p,Q_r(\nu')}(\x),
\quad \x\in\C^d.$$
Moreover, the reducing operators satisfy
$$\left\|\R_{Q_r(\nu)}\R^{-1}_{Q_r(\nu')}\right\|_{\op}\leq\sqrt{d}\left(3^{2n}[\rho]_{\A_{p,3r}}\right)^{\frac{\sqrt{2n}}{r}|\nu-\nu'|},
\quad \nu,\nu'\in r\mathbb{Z}^{2n}.$$
\end{lemma}
\begin{proof}
Let $\Gamma(\nu,\nu')=(a_0,a_1,\cdots,a_k)$ be the discrete path in $r\mathbb{Z}^{2n}$ from $\nu$ to $\nu'$ defined in \cite{Is11}. Then $a_0=\nu$, $a_k=\nu'$, $k\leq\sqrt{2n}|\nu-\nu'|/r$, and $Q_r(a_{j-1})\subset Q_{3r}(a_j)$ for each $j\in\{1,\cdots,k\}$. Consequently, for $\x\in\C^d$, Lemma \ref{3Q} yields that
$$\frac{\rho_{p,Q_r(\nu)}(\x)}{\rho_{p,Q_r(\nu')}(\x)}=\prod_{j=1}^k\frac{\rho_{p,Q_r(a_{j-1})}(\x)}{\rho_{p,Q_r(a_j)}(\x)}
\leq\prod_{j=1}^k\frac{3^{2n/p}\rho_{p,Q_{3r}(a_{j})}(\x)}{\rho_{p,Q_r(a_j)}(\x)}
\leq\left(3^{2n}[\rho]_{\A_{p,3r}}\right)^{\frac{\sqrt{2n}}{r}|\nu-\nu'|}.$$
The second assertion is a direct consequence of the first one and \eqref{reduce}.
\end{proof}

We now give the proof of the implication (d)$\Longrightarrow$(b) of Theorem \ref{main}.
\begin{proof}[Proof of Theorem \ref{main}. (d)$\Longrightarrow$(b)]
Suppose $W\in\A_{p,r_0}$ and write $r=r_0/3$. Then an elementary computation shows that $W\in\A_{p,r}$, and $[W]_{\A_{p,r}}\leq3^{2n}[W]_{\A_{p,r_0}}$; see the proof of Proposition \ref{direct}. Let $\rho$ be the metric defined by
$$\rho_z(\x)=\left|W(z)\x\right|,\quad \x\in\C^d,\quad z\in\C^n.$$
Then $\rho$ is an $\A_{p,3r}$-metric. The proof will be accomplished by a duality argument. To this end, choose $f\in L^p_{\alpha}(\C^n;\C^d)$ and $g\in L^{p'}_{\alpha}(\C^n;\C)$ with
$$\|f\|_{L^p_{\alpha}(\C^n;\C^d)}=\|g\|_{L^{p'}_{\alpha}(\C^n;\C)}=1.$$
Then
\begin{align}\label{pfg}
&\left|\left\langle P^+_{\alpha,W}(f),g\right\rangle_{\alpha}\right|\nonumber\\
=&\left|\int_{\C^n}\int_{\C^n}\left|W(z)W^{-1}(u)f(u)\right|
  \left|K^{\alpha}_z(u)\right|d\lambda_{\alpha}(u)
  \overline{g(z)}e^{-\alpha|z|^2}dv(z)\right|\nonumber\\
\lesssim&\int_{\C^n}\int_{\C^n}
  \left|W(z)W^{-1}(u)f(u)\right||g(z)|
  e^{-\frac{\alpha}{2}|u|^2-\frac{\alpha}{2}|z|^2}e^{-\frac{\alpha}{2}|z-u|^2}dv(u)dv(z)\nonumber\\
\lesssim&\sum_{\nu,\nu'\in r\mathbb{Z}^{2n}}e^{-\frac{\alpha}{4}|\nu-\nu'|^2}\times\nonumber\\
&\qquad \int_{Q_r(\nu)}\int_{Q_r(\nu')}\left|W(z)W^{-1}(u)f(u)\right|
  |g(z)|e^{-\frac{\alpha}{2}|u|^2-\frac{\alpha}{2}|z|^2}dv(u)dv(z)\nonumber\\
\leq&\sum_{\nu,\nu'\in r\mathbb{Z}^{2n}}e^{-\frac{\alpha}{4}|\nu-\nu'|^2}
  \int_{Q_r(\nu')}\left|\R_{Q_r(\nu')}W^{-1}(u)f(u)\right|e^{-\frac{\alpha}{2}|u|^2}dv(u)\times
  \nonumber\\
&\qquad  \int_{Q_r(\nu)}\left\|W(z)\R^{-1}_{Q_r(\nu')}\right\|_{\op}
  |g(z)|e^{-\frac{\alpha}{2}|z|^2}dv(z),
\end{align}
where $\R_{Q_r(\nu')}$ is the reducing operator of $\rho_{p,Q_r(\nu')}$ as in \eqref{reduce}. We now separate into three cases: $1<p<\infty$, $p=1$ and $p=\infty$.

{\bf Case 1: $1<p<\infty$.} Using H\"{o}lder's inequality in \eqref{pfg}, we obtain that
\begin{align}\label{case1}
&\nonumber\left|\left\langle P^+_{\alpha,W}(f),g\right\rangle_{\alpha}\right|\\
\nonumber\lesssim&\left(\sum_{\nu,\nu'\in r\mathbb{Z}^{2n}}e^{-\frac{\alpha}{4}|\nu-\nu'|^2}
    \left(\int_{Q_r(\nu')}\left\|\R_{Q_r(\nu')}W^{-1}(u)\right\|_{\op}|f(u)|e^{-\frac{\alpha}{2}|u|^2}dv(u)\right)^p\right)^{1/p}\\
&\nonumber\quad\times\left(\sum_{\nu,\nu'\in r\mathbb{Z}^{2n}}e^{-\frac{\alpha}{4}|\nu-\nu'|^2}
\left(\int_{Q_r(\nu)}\left\|W(z)\R^{-1}_{Q_r(\nu')}\right\|_{\op}|g(z)|e^{-\frac{\alpha}{2}|z|^2}dv(z)\right)^{p'}\right)^{1/p'}\\
=&:\mathcal{S}_1(f)^{1/p}\cdot\mathcal{S}_2(g)^{1/p'}.
\end{align}
For the term $\mathcal{S}_1(f)$, we have
\begin{align}\label{S1ff}
\nonumber\mathcal{S}_1(f)&=\sum_{\nu,\nu'\in r\mathbb{Z}^{2n}}e^{-\frac{\alpha}{4}|\nu-\nu'|^2}
  \left(\int_{Q_r(\nu')}\left\|\R_{Q_r(\nu')}W^{-1}(u)\right\|_{\op}|f(u)|e^{-\frac{\alpha}{2}|u|^2}dv(u)\right)^p\\
&\lesssim\sum_{\nu'\in r\mathbb{Z}^{2n}}\|\chi_{Q_r(\nu')}f\|^p_{L^p_{\alpha}(\C^n;\C^d)}
  \left(\int_{Q_r(\nu')}\left\|\R_{Q_r(\nu')}W^{-1}(u)\right\|_{\op}^{p'}dv(u)\right)^{p/p'}.
\end{align}
Let $\{\mathbf{e}_j\}_{1\leq j\leq d}$ be the standard orthonormal basis of $\C^d$. Then it is easy to see that for any $d\times d$ matrix $M$, the operator norm of $M$ satisfies
$$\|M\|_{\op}\leq d^{1/2}\max_{1\leq j\leq d}|M\mathbf{e}_j|.$$
Therefore, for $\nu'\in r\mathbb{Z}^{2n}$, we can establish that
\begin{align}\label{matrix-norm}
&\nonumber\int_{Q_r(\nu')}\left\|\R_{Q_r(\nu')}W^{-1}(u)\right\|_{\op}^{p'}dv(u)\\
\nonumber=&\int_{Q_r(\nu')}\left\|W^{-1}(u)\R_{Q_r(\nu')}\right\|_{\op}^{p'}dv(u)\\
\nonumber\leq&\int_{Q_r(\nu')}\left(d^{1/2}\max_{1\leq j\leq d}\left|W^{-1}(u)\R_{Q_r(\nu')}\mathbf{e}_j\right|\right)^{p'}dv(u)\\
\nonumber\leq&\, d^{p'/2}\sum_{j=1}^d\int_{Q_r(\nu')}\left|W^{-1}(u)\R_{Q_r(\nu')}\mathbf{e}_j\right|^{p'}dv(u)\\
\asymp&\, d^{p'/2}\sum_{j=1}^d\Big(\rho^*_{p',Q_r(\nu')}\left(\R_{Q_r(\nu')}\mathbf{e}_j\right)\Big)^{p'},
\end{align}
which, in conjunction with \eqref{reduce-s} and \eqref{R-Apr}, implies that
\begin{align*}
\int_{Q_r(\nu')}\left\|\R_{Q_r(\nu')}W^{-1}(u)\right\|_{\op}^{p'}dv(u)
&\lesssim d^{p'/2}\sum_{j=1}^d\left|\R^{\star}_{Q_r(\nu')}\R_{Q_r(\nu')}\mathbf{e}_j\right|^{p'}\\
&\leq d^{\frac{p'}{2}+1}\left\|\R^{\star}_{Q_r(\nu')}\R_{Q_r(\nu')}\right\|_{\op}^{p'}\\
&\leq d^{\frac{3p'}{2}+1}[W]_{\A_{p,r}}^{p'}\\
&\lesssim d^{\frac{3p'}{2}+1}[W]_{\A_{p,r_0}}^{p'}.
\end{align*}
Combining this with \eqref{S1ff}, we establish that
\begin{equation}\label{S1f}
\mathcal{S}_1(f)\lesssim d^{\frac{3p}{2}+\frac{p}{p'}}[W]_{\A_{p,r_0}}^p.
\end{equation}
We now turn to the term $\mathcal{S}_2(g)$. H\"{o}lder's inequality yields that
\begin{align}\label{S2gg}
\nonumber\mathcal{S}_2(g)&=\sum_{\nu,\nu'\in r\mathbb{Z}^{2n}}e^{-\frac{\alpha}{4}|\nu-\nu'|^2}
  \left(\int_{Q_r(\nu)}\left\|W(z)\R^{-1}_{Q_r(\nu')}\right\|_{\op}|g(z)|e^{-\frac{\alpha}{2}|z|^2}dv(z)\right)^{p'}\\
&\leq\sum_{\nu,\nu'\in r\mathbb{Z}^{2n}}e^{-\frac{\alpha}{4}|\nu-\nu'|^2}\|\chi_{Q_r(\nu)}g\|^{p'}_{L^{p'}_{\alpha}(\C^n;\C)}
  \left(\int_{Q_r(\nu)}\left\|W(z)\R^{-1}_{Q_r(\nu')}\right\|_{\op}^pdv(z)\right)^{\frac{p'}{p}}.
\end{align}
For $\nu,\nu'\in r\mathbb{Z}^{2n}$, using the same method as in \eqref{matrix-norm}, we have
$$\int_{Q_r(\nu)}\left\|W(z)\R^{-1}_{Q_r(\nu')}\right\|_{\op}^pdv(z)
\lesssim d^{p/2}\sum_{j=1}^d\bigg(\rho_{p,Q_r(\nu)}\left(\R^{-1}_{Q_r(\nu')}\mathbf{e}_j\right)\bigg)^p,$$
which, together with \eqref{reduce} and Lemma \ref{nunu'}, implies that
\begin{align*}
\int_{Q_r(\nu)}\left\|W(z)\R^{-1}_{Q_r(\nu')}\right\|_{\op}^pdv(z)
&\lesssim d^{\frac{p}{2}+1}\left\|\R_{Q_r(\nu)}\R^{-1}_{Q_r(\nu')}\right\|_{\op}^p\\
&\leq d^{p+1}\left(3^{2n}[W]_{\A_{p,3r}}\right)^{\frac{p\sqrt{2n}}{r}|\nu-\nu'|}.
\end{align*}
Note that for any $c>1$,
\begin{align*}
\sum_{\nu\in r\mathbb{Z}^{2n}}e^{-\frac{\alpha}{4}|\nu|^2}c^{|\nu|}
&\asymp\sum_{\nu\in r\mathbb{Z}^{2n}}\int_{Q_r(\nu)}e^{-\frac{\alpha}{4}|\nu|^2}c^{|\nu|}dv(z)\\
&\lesssim c^{\frac{\sqrt{2n}}{2}r}\int_{\C^n}e^{-\frac{\alpha}{8}|z|^2}c^{|z|}dv(z)\\
&\asymp c^{\frac{\sqrt{2n}}{2}r}\int_0^{+\infty}t^{2n-1}e^{-\frac{\alpha}{8}t^2+t\log c}dt\\
&=c^{\frac{\sqrt{2n}}{2}r+\frac{2}{\alpha}\log c}\int_{-\frac{4}{\alpha}\log c}^{+\infty}
	\left(t+\frac{4}{\alpha}\log c\right)^{2n-1}e^{-\frac{\alpha}{8}t^2}dt\\
&\lesssim c^{\frac{\sqrt{2n}}{2}r+\frac{2}{\alpha}\log c}\left(1+\log^{2n-1} c\right),
\end{align*}
where the implicit constants depend only on $\alpha$, $r$ and $n$. Consequently, for any $\nu\in r\mathbb{Z}^{2n}$,
\begin{align*}
&\sum_{\nu'\in r\mathbb{Z}^{2n}}e^{-\frac{\alpha}{4}|\nu-\nu'|^2}
  \left(\int_{Q_r(\nu)}\left\|W(z)\R^{-1}_{Q_r(\nu')}\right\|_{\op}^pdv(z)\right)^{p'/p}\\
\lesssim&\, d^{p'+\frac{p'}{p}}\sum_{\nu'\in r\mathbb{Z}^{2n}}e^{-\frac{\alpha}{4}|\nu-\nu'|^2}
  \left(3^{2n}[W]_{\A_{p,3r}}\right)^{\frac{p'\sqrt{2n}}{r}|\nu-\nu'|}\\
\lesssim&\, d^{p'+\frac{p'}{p}}[W]_{\A_{p,r_0}}^{p'n+\frac{144p'^2n^2}{\alpha r_0^2}\log3+\frac{36p'^2n}{\alpha r_0^2}\log[W]_{\A_{p,r_0}}}
  \left(1+\log^{2n-1}[W]_{\A_{p,r_0}}\right),
\end{align*}
which, together with \eqref{S2gg}, yields that
\begin{align*}
\mathcal{S}_2(g)\lesssim
 d^{p'+\frac{p'}{p}}[W]_{\A_{p,r_0}}^{p'n+\frac{144p'^2n^2}{\alpha r_0^2}\log3+\frac{36p'^2n}{\alpha r_0^2}\log[W]_{\A_{p,r_0}}}
\left(1+\log^{2n-1}[W]_{\A_{p,r_0}}\right).
\end{align*}
Inserting this and \eqref{S1f} into \eqref{case1}, we conclude that $P^+_{\alpha,W}:L^p_{\alpha}(\C^n;\C^d)\to L^p_{\alpha}(\C^n;\C)$ is bounded, and
$$\left\|P^+_{\alpha,W}\right\|\lesssim
d^{7/2}[W]_{\A_{p,r_0}}^{1+n+\frac{144p'n^2}{\alpha r_0^2}\log3+\frac{36p'n}{\alpha r_0^2}\log[W]_{\A_{p,r_0}}}
  \left(1+\log^{2n-1}[W]_{\A_{p,r_0}}\right)^{1/p'},$$
where the implicit constant depends only on $\alpha,p,r_0$ and $n$.

{\bf Case 2: $p=1$.} By \eqref{pfg}, we now have
\begin{align*}
&\left|\left\langle P^+_{\alpha,W}(f),g\right\rangle_{\alpha}\right|\\
\lesssim&\sum_{\nu,\nu'\in r\mathbb{Z}^{2n}}e^{-\frac{\alpha}{4}|\nu-\nu'|^2}
  \int_{Q_r(\nu')}\left|\R_{Q_r(\nu')}W^{-1}(u)f(u)\right|e^{-\frac{\alpha}{2}|u|^2}dv(u)\times\\
&\qquad\qquad\int_{Q_r(\nu)}\left\|W(z)\R^{-1}_{Q_r(\nu')}\right\|_{\op}
  |g(z)|e^{-\frac{\alpha}{2}|z|^2}dv(z)\\
\leq&\left(\sum_{\nu,\nu'\in r\mathbb{Z}^{2n}}e^{-\frac{\alpha}{8}|\nu-\nu'|^2}
  \int_{Q_r(\nu')}\left\|\R_{Q_r(\nu')}W^{-1}(u)\right\|_{\op}|f(u)|
  e^{-\frac{\alpha}{2}|u|^2}dv(u)\right)\times\\
&\qquad\qquad\left(\sup_{\nu,\nu'\in r\mathbb{Z}^{2n}}e^{-\frac{\alpha}{8}|\nu-\nu'|^2}
  \int_{Q_r(\nu)}\left\|W(z)\R^{-1}_{Q_r(\nu')}\right\|_{\op}
  |g(z)|e^{-\frac{\alpha}{2}|z|^2}dv(z)\right)\\
=&:\mathcal{S}_3(f)\cdot\mathcal{S}_4(g).
\end{align*}
For the term $\mathcal{S}_3(f)$, similarly as before,
\begin{align*}
\mathcal{S}_3(f)&\lesssim\sum_{\nu'\in r\mathbb{Z}^{2n}}
    \int_{Q_r(\nu')}\left\|\R_{Q_r(\nu')}W^{-1}(u)\right\|_{\op}|f(u)|e^{-\frac{\alpha}{2}|u|^2}dv(u)\\
&\leq\sum_{\nu'\in r\mathbb{Z}^{2n}}\mathrm{ess}\sup_{u\in Q_r(\nu')}\left\|W^{-1}(u)\R_{Q_r(\nu')}\right\|_{\op}
    \int_{Q_r(\nu')}|f(u)|e^{-\frac{\alpha}{2}|u|^2}dv(u)\\
&\leq d^{1/2}\sum_{\nu'\in r\mathbb{Z}^{2n}}\sum_{j=1}^d\rho^*_{\infty,Q_r(\nu')}\left(\R_{Q_r(\nu')}\mathbf{e}_j\right)
    \int_{Q_r(\nu')}|f(u)|e^{-\frac{\alpha}{2}|u|^2}dv(u)\\
&\leq d^{3/2}\sum_{\nu'\in r\mathbb{Z}^{2n}}\left\|\R^{\star}_{Q_r(\nu')}\R_{Q_r(\nu')}\right\|_{\op}
    \int_{Q_r(\nu')}|f(u)|e^{-\frac{\alpha}{2}|u|^2}dv(u)\\
&\leq d^{5/2}[W]_{\A_{1,r}}\\
&\lesssim d^{5/2}[W]_{\A_{1,r_0}}.
\end{align*}
For the term $\mathcal{S}_4(g)$, we have that
\begin{align*}
\mathcal{S}_4(g)\leq\sup_{\nu,\nu'\in r\mathbb{Z}^{2n}}e^{-\frac{\alpha}{8}|\nu-\nu'|^2}
    \int_{Q_r(\nu)}\left\|W(z)\R^{-1}_{Q_r(\nu')}\right\|_{\op}dv(z).
\end{align*}
Noting that for $c>0$, $\sup_{t\in\mathbb{R}}e^{-\frac{\alpha}{8}t^2}c^t=c^{\frac{2}{\alpha}\log c}$, we may apply \eqref{reduce}, Lemma \ref{nunu'} and the same method as in \eqref{matrix-norm} to obtain that
\begin{align*}
\mathcal{S}_4(g)
&\lesssim\, d^{1/2}\sup_{\nu,\nu'\in r\mathbb{Z}^{2n}}e^{-\frac{\alpha}{8}|\nu-\nu'|^2}
  \sum_{j=1}^d\rho_{1,Q_r(\nu)}\left(\R^{-1}_{Q_r(\nu')}\mathbf{e}_j\right)\\
&\leq d^{3/2}\sup_{\nu,\nu'\in r\mathbb{Z}^{2n}}e^{-\frac{\alpha}{8}|\nu-\nu'|^2}\left\|\R_{Q_r(\nu)}\R^{-1}_{Q_r(\nu')}\right\|_{\op}\\
&\leq d^2\sup_{\nu,\nu'\in r\mathbb{Z}^{2n}}e^{-\frac{\alpha}{8}|\nu-\nu'|^2}
  \left(3^{2n}[W]_{\A_{1,3r}}\right)^{\frac{\sqrt{2n}}{r}|\nu-\nu'|}\\
&\lesssim d^2[W]_{\A_{1,r_0}}^{\frac{144n^2}{\alpha r_0^2}\log3+\frac{36n}{\alpha r_0^2}\log[W]_{\A_{1,r_0}}}.
\end{align*}
Therefore, $P^+_{\alpha,W}:L^1_{\alpha}(\C^n;\C^d)\to L^1_{\alpha}(\C^n;\C)$ is bounded, and
$$\left\|P^+_{\alpha,W}\right\|\lesssim
d^{9/2}[W]_{\A_{1,r_0}}^{1+\frac{144n^2}{\alpha r_0^2}\log3+\frac{36n}{\alpha r_0^2}\log[W]_{\A_{1,r_0}}}$$
with implicit constant depends only on $\alpha,r_0$ and $n$.

{\bf Case 3: $p=\infty$.} We deduce from \eqref{pfg} that
\begin{align*}
&\left|\left\langle P^+_{\alpha,W}(f),g\right\rangle_{\alpha}\right|\\
	\lesssim&\left(\sup_{\nu,\nu'\in r\mathbb{Z}^{2n}}e^{-\frac{\alpha}{8}|\nu-\nu'|^2}
	\int_{Q_r(\nu')}\left\|\R_{Q_r(\nu')}W^{-1}(u)\right\|_{\op}|f(u)|
	e^{-\frac{\alpha}{2}|u|^2}dv(u)\right)\times\\
	&\qquad\qquad\left(\sum_{\nu,\nu'\in r\mathbb{Z}^{2n}}e^{-\frac{\alpha}{8}|\nu-\nu'|^2}
	\int_{Q_r(\nu)}\left\|W(z)\R^{-1}_{Q_r(\nu')}\right\|_{\op}
	|g(z)|e^{-\frac{\alpha}{2}|z|^2}dv(z)\right)\\
	=&:\mathcal{S}_5(f)\cdot\mathcal{S}_6(g).
\end{align*}
Arguing as before, we can establish that
$$\mathcal{S}_5(f)\lesssim d^{5/2}[W]_{\A_{\infty,r_0}}$$
and
$$\mathcal{S}_6(g)\lesssim d^2[W]_{\A_{\infty,r_0}}^{n+\frac{288n^2}{\alpha r_0^2}\log3+\frac{72n}{\alpha r_0^2}\log[W]_{\A_{\infty,r_0}}}
\left(1+\log^{2n-1}[W]_{\A_{\infty,r_0}}\right).$$
Therefore, $P^+_{\alpha,W}:L^{\infty}_{\alpha}(\C^n;\C^d)\to L^{\infty}_{\alpha}(\C^n;\C)$ is bounded, and
$$\left\|P^+_{\alpha,W}\right\|\lesssim
d^{9/2}[W]_{\A_{\infty,r_0}}^{1+n+\frac{288n^2}{\alpha r_0^2}\log3+\frac{72n}{\alpha r_0^2}\log[W]_{\A_{\infty,r_0}}}
\left(1+\log^{2n-1}[W]_{\A_{\infty,r_0}}\right)$$
with implicit constant depends only on $\alpha,r_0$ and $n$. The proof is complete.
\end{proof}

As stated before, it follows from Theorem \ref{main} that for each $1\leq p\leq\infty$, all the classes of $\A_{p,r}$-weights coincide for $r>0$. We end this paper by a quantitative description of this fact, which is independent of Theorem \ref{main} and has its own interest.

\begin{proposition}\label{direct}
Let $1\leq p\leq\infty$, $0<r_1<r_2<\infty$, and let $\rho:z\mapsto \rho_z$ be a metric on $\C^n$. Then $\rho$ is an $\A_{p,r_1}$-metric if and only if it is an $\A_{p,r_2}$-metric. Moreover,
$$\left(\frac{r_1}{r_2}\right)^{2n}[\rho]_{\A_{p,r_1}}\leq[\rho]_{\A_{p,r_2}}
\leq3^{4n^2\left(1+\frac{3r_2}{r_1}\right)}d^{\frac{5}{2}}
\left(2\sqrt{\frac{r_1}{r_2}}+3\sqrt{\frac{r_2}{r_1}}\right)^{4n}
[\rho]_{\A_{p,r_1}}^{1+2n\left(1+\frac{3r_2}{r_1}\right)}.$$
In particular, for each $1\leq p\leq\infty$, all the classes of $\A_{p,r}$-metrics coincide for $r>0$.
\end{proposition}
\begin{proof}
Suppose first that $\rho$ is an $\A_{p,r_2}$-metric. Then for any $z\in\C^n$ and $\x\in\C^d$, noting that
\begin{align*}
\left(\rho_{p,Q_{r_2}(z)}\right)^*(\x)
&=\sup_{\y\in\C^d\setminus\{0\}}\frac{|\langle\x,\y\rangle|}{\rho_{p,Q_{r_2}(z)}(\y)}\\
&\leq\left(\frac{r_2}{r_1}\right)^{2n/p}\sup_{\y\in\C^d\setminus\{0\}}
	\frac{|\langle\x,\y\rangle|}{\rho_{p,Q_{r_1}(z)}(\y)}\\
&=\left(\frac{r_2}{r_1}\right)^{2n/p}\left(\rho_{p,Q_{r_1}(z)}\right)^*(\x),
\end{align*}
we establish that
\begin{align*}
\rho^*_{p',Q_{r_1}(z)}(\x)&\leq\left(\frac{r_2}{r_1}\right)^{2n/p'}\rho^*_{p',Q_{r_2}(z)}(\x)\\
&\leq\left(\frac{r_2}{r_1}\right)^{2n/p'}[\rho]_{\A_{p,r_2}}\left(\rho_{p,Q_{r_2}(z)}\right)^*(\x)\\
&\leq\left(\frac{r_2}{r_1}\right)^{2n}[\rho]_{\A_{p,r_2}}\left(\rho_{p,Q_{r_1}(z)}\right)^*(\x).
\end{align*}
Therefore, $\rho$ is an $\A_{p,r_1}$-metric, and
\begin{equation}\label{r_2-r_1}
[\rho]_{\A_{p,r_1}}\leq\left(\frac{r_2}{r_1}\right)^{2n}[\rho]_{\A_{p,r_2}}.
\end{equation}
	
Suppose now that $\rho$ is an $\A_{p,r_1}$-metric. Fix $z\in\C^n$, and let
$$\Lambda=\left\{\nu\in z+\frac{r_1}{3}\mathbb{Z}^{2n}:Q_{r_1/3}(\nu)\cap Q_{r_2}(z)\neq\emptyset\right\}.$$
We consider the reducing operators $\R_{Q}$ and $\R^{\star}_Q$ of $\rho_{p,Q}$ and $\rho^*_{p',Q}$ respectively. For any $\x\in\C^d$,  \eqref{reduce} implies that
\begin{align*}
\left|\R_{Q_{r_2}(z)}\x\right|&\leq\sqrt{d}\rho_{p,Q_{r_2}(z)}(\x)\\
&\leq\sqrt{d}\left(\frac{r_1}{3r_2}\right)^{2n/p}\sum_{\nu\in\Lambda}\rho_{p,Q_{r_1/3}(\nu)}(\x)\\
&\leq\sqrt{d}\left(\frac{r_1}{3r_2}\right)^{2n/p}\sum_{\nu\in\Lambda}
	\left|\R_{Q_{r_1/3}(\nu)}\x\right|.
\end{align*}
Similarly, \eqref{reduce-s} yields that
$$\left|\R^{\star}_{Q_{r_2}(z)}\x\right|
\leq\sqrt{d}\left(\frac{r_1}{3r_2}\right)^{2n/p'}\sum_{\nu\in\Lambda}
\left|\R^{\star}_{Q_{r_1/3}(\nu)}\x\right|.$$
Consequently,
\begin{align}\label{r_2r_2}
\nonumber\left\|\R_{Q_{r_2}(z)}\R^{\star}_{Q_{r_2}(z)}\right\|_{\op}
&=\sup_{\x\in\s_d}\left|\R_{Q_{r_2}(z)}\R^{\star}_{Q_{r_2}(z)}\x\right|\\
&\nonumber\leq\sqrt{d}\left(\frac{r_1}{3r_2}\right)^{2n/p}
	\sum_{\nu\in\Lambda}\sup_{\x\in\s_d}\left|\R_{Q_{r_1/3}(\nu)}\R^{\star}_{Q_{r_2}(z)}\x\right|\\
&\nonumber=\sqrt{d}\left(\frac{r_1}{3r_2}\right)^{2n/p}\sum_{\nu\in\Lambda}
	\left\|\R_{Q_{r_1/3}(\nu)}\R^{\star}_{Q_{r_2}(z)}\right\|_{\op}\\
&\nonumber=\sqrt{d}\left(\frac{r_1}{3r_2}\right)^{2n/p}\sum_{\nu\in\Lambda}
	\left\|\R^{\star}_{Q_{r_2}(z)}\R_{Q_{r_1/3}(\nu)}\right\|_{\op}\\
&\nonumber=\sqrt{d}\left(\frac{r_1}{3r_2}\right)^{2n/p}
	\sum_{\nu\in\Lambda}\sup_{\x\in\s_d}\left|\R^{\star}_{Q_{r_2}(z)}\R_{Q_{r_1/3}(\nu)}\x\right|\\
&\leq d\left(\frac{r_1}{3r_2}\right)^{2n}\sum_{\nu,\nu'\in\Lambda}
	\left\|\R^{\star}_{Q_{r_1/3}(\nu')}\R_{Q_{r_1/3}(\nu)}\right\|_{\op},
\end{align}
where $\s_d$ is the unit sphere of $\C^d$. Since $\rho$ is an $\A_{p,r_1}$-metric, we can use \eqref{r_2-r_1} to obtain that $[\rho]_{\A_{p,r_1/3}}\leq3^{2n}[\rho]_{\A_{p,r_1}}$. Then it follows from \eqref{R-Apr} that
$$\sup_{u\in\C^n}\left\|\R_{Q_{r_1/3}(u)}\R^{\star}_{Q_{r_1/3}(u)}\right\|_{\op}
\leq3^{2n}d[\rho]_{\A_{p,r_1}},$$
which, in conjunction with Lemma \ref{nunu'}, implies that for any $\nu,\nu'\in\Lambda$,
\begin{align*}
\left\|\R^{\star}_{Q_{r_1/3}(\nu')}\R_{Q_{r_1/3}(\nu)}\right\|_{\op}
&\leq\left\|\R^{\star}_{Q_{r_1/3}(\nu')}\R_{Q_{r_1/3}(\nu')}\right\|_{\op}
    \cdot\left\|\R^{-1}_{Q_{r_1/3}(\nu')}\R_{Q_{r_1/3}(\nu)}\right\|_{\op}\\
&\leq3^{2n}d^{3/2}[\rho]_{\A_{p,r_1}}
\left(3^{2n}[\rho]_{\A_{p,r_1}}\right)^{\frac{3\sqrt{2n}}{r_1}|\nu-\nu'|}.
\end{align*}
Inserting the above estimate into \eqref{r_2r_2} yields that
$$\left\|\R_{Q_{r_2}(z)}\R^{\star}_{Q_{r_2}(z)}\right\|_{\op}
\leq3^{2n}d^{5/2}\left(\frac{r_1}{3r_2}\right)^{2n}[\rho]_{\A_{p,r_1}}
\sum_{\nu,\nu'\in\Lambda}\left(3^{2n}[\rho]_{\A_{p,r_1}}\right)^{\frac{3\sqrt{2n}}{r_1}|\nu-\nu'|}.$$
It is easy to see that for any $\nu,\nu'\in\Lambda$, $|\nu-\nu'|\leq\sqrt{2n}\left(\frac{r_1}{3}+r_2\right)$, and the number $\sharp\Lambda$ of elements in the set $\Lambda$ satisfies $\sharp\Lambda\leq\left(2+\frac{3r_2}{r_1}\right)^{2n}$. Hence we can establish that
$$\left\|\R_{Q_{r_2}(z)}\R^{\star}_{Q_{r_2}(z)}\right\|_{\op}\leq
3^{4n^2\left(1+\frac{3r_2}{r_1}\right)}d^{\frac{5}{2}}
\left(2\sqrt{\frac{r_1}{r_2}}+3\sqrt{\frac{r_2}{r_1}}\right)^{4n}
[\rho]_{\A_{p,r_1}}^{1+2n\left(1+\frac{3r_2}{r_1}\right)}.$$
Since $z\in\C^n$ is arbitrary, in view of \eqref{R-Apr}, we conclude that $\rho$ is an $\A_{p,r_2}$-metric, and
$$[\rho]_{\A_{p,r_2}}\leq
3^{4n^2\left(1+\frac{3r_2}{r_1}\right)}d^{\frac{5}{2}}
\left(2\sqrt{\frac{r_1}{r_2}}+3\sqrt{\frac{r_2}{r_1}}\right)^{4n}
[\rho]_{\A_{p,r_1}}^{1+2n\left(1+\frac{3r_2}{r_1}\right)}.$$
The proof is complete.
\end{proof}

\medskip





\end{document}